\documentclass[english,10pt,reqno,twoside]{amsart}
\usepackage[T1]{fontenc}
\usepackage[latin9]{inputenc}
\usepackage{geometry}
\usepackage{amsthm}
\usepackage{amssymb}
\usepackage{tikz-cd}
\usepackage{theoremref}

\usepackage{bbm}
\usepackage{color}

\makeatletter

\numberwithin{equation}{section}
\numberwithin{figure}{section}
\theoremstyle{plain}
\newtheorem{thm}{\protect\theoremname}
\theoremstyle{plain}

\newtheorem{lem}[thm]{\protect\lemmaname}
\newtheorem{cor}[thm]{\protect\corollaryname}

\theoremstyle{definition}

\newtheorem{defn}[thm]{\protect\defnname}

\newtheorem{rem}{\protect\remarkname}

\makeatother

\usepackage{lmodern}    
\usepackage[english]{babel}

\providecommand{\propositionname}{Proposition}
\providecommand{\lemmaname}{Lemma}
\providecommand{\theoremname}{Theorem}
\providecommand{\corollaryname}{Corollary}
\providecommand{\notationname}{Notation}
\providecommand{\conventionname}{Convention}
\providecommand{\exname}{Example}
\providecommand{\defnname}{Definition}
\providecommand{\questionname}{Question}
\providecommand{\remarkname}{Remark}

\usepackage{graphicx}
\usepackage[all]{xy}
\usepackage{paralist}
\usepackage{subfigure}
\usepackage{multicol}
\usepackage{comment}

\usepackage{xcolor}
\definecolor{Chocolat}{rgb}{0.36, 0.2, 0.09}
\definecolor{BleuTresFonce}{rgb}{0.215, 0.215, 0.86}
\usepackage[colorlinks,final,hyperindex,pagebackref=true]{hyperref}
\hypersetup{
	pdftex,
    linkcolor=blue,
	citecolor=purple,
	filecolor=BleuTresFonce,
	urlcolor=BleuTresFonce
}

\usepackage[capitalise]{cleveref}
\usepackage{placeins}

\usepackage{tikz, tkz-tab, tikz-cd}
\usepackage{mathtools}

\usepackage[new]{old-arrows}

\theoremstyle{plain}
\newtheorem{mainthm}{Theorem}

\numberwithin{equation}{section}

\newcommand{\A}{\mathbf{A}}
\newcommand{\C}{\mathbf{C}}

\newcommand{\G}{\mathbf{G}}
\newcommand{\N}{\mathbf{N}}
\newcommand{\oo}{\mathbf{0}}
\newcommand{\Q}{\mathbb{Q}}
\newcommand{\R}{\mathbf{R}}

\newcommand{\Z}{\mathbf{Z}}

\newcommand{\ba}{\mathbf{a}}
\newcommand{\bb}{\mathbf{b}}

\newcommand{\Tor}{\operatorname{Tor}}
\newcommand{\Spec}{\operatorname{Spec}}
\newcommand{\Sing}{\operatorname{Sing}}

\usetikzlibrary{decorations.markings}
\tikzstyle{vertex}=[circle, draw, inner sep=0pt, minimum size=5pt]

\begin{document}

\title[Generalized Zariski cancellation for Brieskorn--Pham varieties]{Generalized Zariski cancellation for Brieskorn--Pham varieties}

\author[B. Hajra]{Buddhadev Hajra}
\address{Indian Statistical Institute, Kolkata, Stat-Math Unit, 203 B. T. Road, Baranagar, Kolkata - 700108, India}
\email{hajrabuddhadev92@gmail.com}

\author[M. Upmanyu]{Mohit Upmanyu}
\address{Simion Stoilow Institute of Mathematics of the Romanian Academy, 21 Calea Grivitei Street, 010702 Bucharest, Romania}
\email{mohit.i.upmanyu@gmail.com}

\subjclass[2020]{14F35, 14F45, 14J10, 55P20, 55R10}

\keywords{Brieskorn--Pham singularities, Hilbert series, generalized Zariski cancellation, algebroid spaces}

\begin{abstract}
We establish a generalized Zariski cancellation theorem for Brieskorn--Pham varieties over the field of complex numbers. More precisely, we show that if two complex Brieskorn--Pham varieties become isomorphic after taking a product with an arbitrary separated complex scheme having a smooth point, then they are already isomorphic not merely as complex algebraic varieties but, in fact, as $\C^*$-varieties. The proof combines our general cancellation theorem for complex algebraic varieties with a unique singularity, whose proof relies on the analytic cancellation theorem of Hauser--M\"uller, with an exponent rigidity theorem for Brieskorn--Pham varieties. The latter asserts that, over any field of characteristic zero, the exponent tuple appearing in the defining equation completely determines the isomorphism class of the corresponding Brieskorn--Pham variety.
\end{abstract}

\maketitle

\tableofcontents

\section{\bf Introduction}\label{se1}
The central theme of this article concerns the generalized Zariski cancellation problem. Given algebraic varieties $X$, $Y$, and $Z$ over a field $k$, this problem asks whether
\[
X\times Z \cong_k Y\times Z
\implies
X\cong_k Y
\]
as $k$-varieties. The classical Zariski cancellation problem corresponds to the special case $Z=\A^1_k$. The behaviour of cancellation is remarkably subtle. While positive results are known in several situations---for instance, Fujita \cite{Fuj1979} proved cancellation for the affine plane $\A^2_k$ over any field of characteristic zero---the answer is negative in general. In \cite{Dan1989}, Danielewski introduced a family of complex hypersurfaces
\[
D_n:=V(x_1^n x_2+x_3^2+1)\subseteq \A^3_\C,
\]
showing that $D_1\times\A^1_\C\cong_\C D_n\times\A^1_\C$ for every $n\in\N$, i.e., all $\A^1_\C$-cylinders over the $D_n$ are isomorphic. Motivated by Zariski cancellation, one may ask whether the surfaces $D_n$ are pairwise non-isomorphic. Fieseler \cite{Fie1994} answered this affirmatively by distinguishing them via their homology groups at infinity. Consequently, the family $\{D_n\}_{n\in\N}$ provides counterexamples to the classical Zariski cancellation problem.

Over $\C$, cancellation has also been studied in complex analytic geometry with respect to analytic isomorphism equivalence. In the global analytic setting, the answer is negative in general: Shioda \cite{Shi1977} constructed pairwise non-isomorphic elliptic curves $E$, $E_1$, and $E_2$ such that $E_1\times E\cong E_2\times E$. On the other hand, in the category of complete varieties, $X\times Z\cong Y\times Z$ entails $X\cong Y$, provided that $Z$ is projective and either the pair $(X,Z)$ or the pair $(Y,Z)$ is Picard-independent (see \cite{Fuj1981}). By contrast, in the local analytic setting, Hauser and M\"uller \cite[Theorem~1]{HM1990} established generalized cancellation for complex analytic space germs. This result, together with their unique factorization theorem for algebroid spaces (see Section~\ref{subse2.3}), forms a key ingredient of the present work.

The first main result of this article establishes a generalized cancellation theorem for affine varieties admitting good
$\C^*$-actions and having an isolated singularity.

\begin{mainthm}[{cf. Theorem \ref{Thm: Our main cancellation theorem for C*-varieties}}]\label{thm:mainA}
Let $R$ and $S$ be positively graded affine domains over $\C$. Let $V:=\Spec R$ and $W:=\Spec S$ be the corresponding complex affine irreducible varieties admitting good $\C^*$-actions with vertices $p$ and $q$, respectively. Assume that $p$ and $q$ are the unique singularities of $V$ and $W$, respectively. Let $Z$ be a separated scheme over $\C$ (not necessarily connected) having a smooth point such that
\[
V\times Z \cong_{\C} W\times Z.
\]
Then $V$ and $W$ are isomorphic as complex affine varieties (not necessarily as $\C^*$-varieties).
\end{mainthm}

A particularly striking application of Theorem~\ref{thm:mainA} concerns Brieskorn--Pham varieties.

\medskip

Let $k$ be a field of characteristic zero, not necessarily algebraically closed. For $n\geq 2$ and an $n$-tuple $\ba=(a_1,\ldots,a_n)\in(\N_{\geq 2})^n$, consider the affine $k$-algebra
\[
B_{\ba,k}:=
\frac{k[X_1,\ldots,X_n]}{(X_1^{a_1}+\cdots+X_n^{a_n})},
\]
called a \emph{Brieskorn--Pham hypersurface ring over $k$}. The associated affine hypersurface
\[
V_{\ba,k}:=\Spec(B_{\ba,k})\subseteq \A^n_k
\]
is called the \emph{Brieskorn--Pham variety over $k$} corresponding to the exponent tuple $\ba$. We sometimes denote the defining polynomial by $f_\ba$, often referred to as the \emph{Brieskorn polynomial} associated to $\ba$. Throughout this article, we may omit the phrase ``over $k$'' whenever the ground field is clear from the context and denote the corresponding variety simply by $V_\ba$ instead of $V_{\ba,k}$.

By the Jacobian criterion, $V_{\ba,k}$ has a unique singular point at the origin, denoted by $\oo_{\ba,k}$. Moreover, $V_{\ba,k}$ is normal whenever $n\geq 3$, whereas it is non-normal when $n=2$. Brieskorn--Pham varieties need not be irreducible; for example, $V_{(2,2),k}$ is reducible whenever $k\supseteq \Q(i)$.

When $k$ is not algebraically closed, it is often more natural to regard $V_{\ba,k}$ as the affine $k$-scheme $\Spec(B_{\ba,k})$ rather than through its set of $k$-rational points, since the latter may be very small. For instance, if $k=\R$, then
\[
V_{(2,2,\ldots,2),\R}
=\Spec\left(
\frac{\R[X_1,\ldots,X_n]}
{(X_1^2+\cdots+X_n^2)}
\right)
\]
has a unique $\R$-rational point, namely the origin. Accordingly, throughout this article, by a \emph{variety} over a field $k$ we mean a reduced separated scheme of finite type over $k$, and irreducibility will be imposed only when explicitly stated. In particular, we shall continue to refer to $\Spec(B_{\ba,k})$ as a Brieskorn--Pham variety.

When $k=\C$, Brieskorn--Pham varieties form a classical family of isolated hypersurface singularities that have been extensively studied from topological, differential-geometric, analytic, metric, and motivic viewpoints. Their links, Milnor fibers, monodromy operators, mixed Hodge structures, and related invariants play prominent roles in singularity theory.

Our next main result shows that Brieskorn--Pham varieties satisfy a strong form of generalized Zariski cancellation. Indeed, cancellation by any complex algebraic variety possessing a smooth point completely determines the exponent tuple. More precisely:

\begin{mainthm}[{cf. Theorem~\ref{Thm: Our main cancellation theorem for B-P varieties}}]\label{thm:mainB}
Let $n\geq 2$, and let $V_\ba$ and $V_\bb$ be two Brieskorn--Pham varieties defined over $\C$ corresponding to two tuples
$\ba,\bb\in(\N_{\geq 2})^n$.
Then, for any separated $\C$-scheme $Z$ (not necessarily connected) having a smooth point, the following are equivalent:
\begin{enumerate}[\indent\rm(1)]
    \item $V_\ba\times Z\cong_{\C} V_\bb\times Z$;
    \item $V_\ba\cong_{\C} V_\bb$;
    \item $V_\ba$ and $V_\bb$ are isomorphic as $\C^*$-varieties;
    \item $\ba \sim \bb$;
\end{enumerate}
where $\ba\sim\bb$ denotes equality up to permutation of entries.
\end{mainthm}

When $k=\C$, we also consider the associated analytic germ and formal completion at the singular point $\oo_{\ba}$, namely
\[
V^{\rm an}_\ba:=(V_\ba,\oo_\ba)
\quad \text{and} \quad
\widehat{V}_\ba:=\widehat{(V_\ba,\oo_\ba)},
\]
respectively defined by the local algebras
\[
B^{\rm an}_\ba
=
\frac{\C\{X_1,\ldots,X_n\}}{(X_1^{a_1}+\cdots+X_n^{a_n})}
\quad \text{and} \quad
\widehat{B}_\ba
=
\frac{\C[\![X_1,\ldots,X_n]\!]}{(X_1^{a_1}+\cdots+X_n^{a_n})}.
\]

The extent to which the exponent tuple can be recovered from geometric structures associated with a Brieskorn--Pham variety is a natural and subtle question. Several rigidity phenomena of this type are already known. According to Milnor \cite{Mil1968}, every isolated hypersurface singularity possesses a well-defined embedded topological type and an associated algebraic knot. For complex Brieskorn--Pham singularities, a theorem of Yoshinaga--Suzuki \cite{YS1978} (see also \cite[Corollary]{Yos1983}) shows that the topological type determines the exponent tuple: if the germs $(V_{\ba},\oo_{\ba})$ and $(V_{\bb},\oo_{\bb})$ are homeomorphic, then $\ba\sim\bb$. In \cite{BS2011}, the authors proved that two Brieskorn polynomials $f_\ba$ and $f_\bb$ have cobordant algebraic knots if and only if $\ba\sim\bb$, provided that no exponent is a multiple of another for either polynomial. More recently, Fern{\'a}ndez--Jelonek--Sampaio \cite{FJS2024_arxiv} showed that bi-Lipschitz equivalence of the analytic germs $(V_{\ba},\oo_{\ba})$ and $(V_{\bb},\oo_{\bb})$ forces the smallest exponents of $\ba$ and $\bb$ to coincide and occur with the same multiplicity. Over the real numbers, Campesato \cite{Cam2018} proved that the arc-analytic equivalence class of a Brieskorn polynomial determines its exponents, describing this result as an analogue of the Yoshinaga--Suzuki theorem over $\R$.

These results should be contrasted with the behaviour of substantially weaker invariants. The topology of the link alone does not determine the exponent tuple. Indeed, Yoshinaga exhibited infinite families of Brieskorn--Pham singularities with pairwise distinct exponent tuples whose links are all homeomorphic to topological spheres (see \cite[Examples~1,~2]{Yos1983}). Since the smooth locus $V_{\ba}^{\mathrm{sm}}:=V_{\ba}\setminus \{\oo_{\ba}\}$ is diffeomorphic to $L_{\ba}\times(0,\infty)$, where $L_{\ba}$ denotes the link of the singularity $(V_{\ba},\oo_{\ba})$, the smooth locus and the link have the same homotopy type. Consequently, neither the topology of the link nor the homotopy type of the smooth locus determines the exponent tuple in general.

Motivated by these rigidity phenomena, we show that the exponent tuple is also completely determined by the algebraic, analytic, and formal structures naturally associated with a Brieskorn--Pham variety. More precisely, a key ingredient in the proof of Theorem~\ref{thm:mainB} is the following exponent rigidity theorem.

\begin{mainthm}[{cf. Theorem~\ref{Thm: Exponent rigidity of B-P variety}, Corollary~\ref{Cor: Exponent rigidity of B-P variety in analytic setting over complex numbers}}]\label{thm:mainC} Let $n\geq 2$ and let $\ba,\bb\in(\N_{\geq 2})^n$. \begin{enumerate}[\indent\rm(1)] \item Over a field $k$ of characteristic $0$, \[ B_{\ba,k} \cong_{k\text{-alg}} B_{\bb,k} \iff V_{\ba,k} \cong_k V_{\bb,k} \iff \ba\sim\bb, \] where $\ba\sim\bb$ denotes equality up to permutation of entries. \item If $k=\C$, then \[ \widehat{B}_\ba \cong_{\C\text{-alg}} \widehat{B}_\bb \iff \widehat{V}_\ba \cong_{\rm bihol} \widehat{V}_\bb \iff V_\ba^{\rm an} \cong_{\rm bihol} V_\bb^{\rm an} \iff \ba\sim\bb. \] \end{enumerate} \end{mainthm}

\subsubsection{Strategy of the proofs and organization of the paper}

The paper is organized as follows. Section~\ref{se2} collects the necessary preliminaries concerning Brieskorn--Pham varieties, the good $\G_{m,k}$-actions they admit, complex analytic germs and algebroid spaces, and the analytic cancellation theorem of Hauser--M\"uller. The proofs of the main results are presented in Section~\ref{se3} according to their logical dependencies rather than the order in which the results are stated in the Introduction.

We begin in Subsection~\ref{subse3.1} by proving exponent rigidity for Brieskorn--Pham varieties over an arbitrary field of characteristic zero (Theorem~\ref{thm:mainC}). This is achieved by studying certain zero-dimensional complete intersection schemes arising from the Jacobian ideals of Brieskorn--Pham hypersurfaces and showing, via a Hilbert series argument, that the associated coordinate rings determine the exponent tuple.

Next, in Subsection~\ref{subse3.2}, we establish an analytic reduction theorem for cylinder isomorphisms. More precisely, we show that if two complex algebraic varieties with a unique singular point become isomorphic after taking a product with an arbitrary complex algebraic variety having a smooth point, then the corresponding complex analytic germs at their unique singular points are analytically isomorphic. The proof relies on the analytic cancellation and unique factorization theorems of Hauser--M\"uller for algebroid spaces, together with Artin's approximation theorem. We then combine this result with a theorem of R. V. Gurjar to establish a generalized Zariski cancellation theorem for complex affine irreducible varieties admitting good $\C^*$-actions and having a unique singularity (Theorem~\ref{thm:mainA}).

Finally, in Subsection~\ref{subse3.3}, we establish the generalized Zariski cancellation theorem for Brieskorn--Pham varieties (Theorem~\ref{thm:mainB}) by combining the exponent rigidity theorem with the analytic reduction theorem established earlier. This theorem constitutes the main culmination of the paper.

\section{\bf Preliminaries}\label{se2}

\subsection{Good $\G_{m,k}$-actions on Brieskorn--Pham varieties}\label{subse2.1}

For a field $k$, let $\G_{m,k}:=\Spec k[T,T^{-1}]$ denote the multiplicative group scheme over $k$. Recall that an effective $\G_{m,k}$-action on an affine variety $V=\Spec A$ is called \emph{good} if there exists a unique fixed point $v\in V$, called the \emph{vertex}, such that $v$ belongs to the closure of every $\G_{m,k}$-orbit. Since a $\G_{m,k}$-action on $V$ is equivalent to a $\Z$-grading of the coordinate ring $A$, the action is good if and only if the induced grading $A=\bigoplus_{d\in\Z}A_d$ is positive and satisfies $A_0=k$. Equivalently, the irrelevant ideal $A_+:=\bigoplus_{d>0}A_d$ is the unique homogeneous maximal ideal of $A$, corresponding to the fixed point $v$.

Let $k$ be an arbitrary field, let $n\geq 2$, and let $\ba=(a_1,\ldots,a_n)\in(\N_{\geq 2})^n$. Put
\[
N:=\operatorname{lcm}(a_1,\ldots,a_n)
\qquad\text{and}\qquad
w_i:=\frac{N}{a_i}
\quad (1\leq i\leq n).
\]
Then $B_\ba$ admits a positive $\Z$-grading determined by $\deg(X_i)=w_i$. Since
\[
\deg(X_i^{a_i})=a_iw_i=N
\]
for every $i$, the polynomial $f_\ba=X_1^{a_1}+\cdots+X_n^{a_n}$ is homogeneous of degree $N$. Consequently, this grading induces a $\G_{m,k}$-action on $V_\ba$ given by
\[
t\cdot(x_1,\ldots,x_n)
= (t^{w_1}x_1,\ldots,t^{w_n}x_n),
\qquad
t\in\G_{m,k}.
\]

Since all weights $w_i$ are strictly positive, the induced grading on $B_\ba$ is positive. Moreover, the only homogeneous elements of degree $0$ are the constants, so $(B_\ba)_0=k$. Its unique fixed point is the vertex $\oo_\ba\in V_\ba$, corresponding to the homogeneous maximal ideal $(X_1,\ldots,X_n)=(B_\ba)_+$. Furthermore, for every point $x=(x_1,\ldots,x_n)\in V_\ba$, the orbit map
\[
\G_{m,k}\longrightarrow V_\ba,
\qquad
t\longmapsto
(t^{w_1}x_1,\ldots,t^{w_n}x_n),
\]
extends to a morphism $\A^1_k\to V_\ba$ by sending $0$ to $\oo_\ba$, since $w_i>0$ for all $i$. Consequently, $\oo_\ba$ belongs to the closure of every $\G_{m,k}$-orbit. Under our standing assumptions, $\oo_\ba$ is the unique singular point of $V_\ba$. Since $\oo_\ba$ is the unique fixed point of the action, the above $\G_{m,k}$-action on $V_\ba$ is good.\\

Since the proofs rely heavily on complex analytic methods, we briefly recall some notions from complex analytic geometry that will be used throughout this section.

\subsection{A quick glimpse of complex analytic geometry}\label{subse2.2}
An \emph{analytic space} is a locally ringed space $(X,\mathcal{H}_X)$ locally isomorphic to a closed complex analytic subspace of an open subset of $\C^n$ under the usual complex Euclidean topology. If $(X,\mathcal{H}_X)$ is an analytic space and $x \in X$, then the pair $(X,x)$ is called an \emph{analytic germ} modulo the equivalence that two such pairs $(X,x)$ and $(Y,y)$ are said to be \emph{analytically isomorphic} (or \emph{analytically equivalent}) if there exist open neighbourhoods $M(x) \subseteq X$ of $x$ and $N(y) \subseteq Y$ of $y$ under the usual complex topology, together with an isomorphism of analytic spaces
\[
\varphi : M(x) \xrightarrow{\sim} N(y)
\]
such that $\varphi(x)=y$.

Now let $X$ be a complex algebraic variety. Then $X$ has the structure sheaf $\mathcal{O}_X$ under the Zariski topology. On the other hand, associated to $X$ is a natural analytic space $(X^{\rm an},\mathcal{O}^{\rm an}_X)$ obtained by endowing $X$ with the sheaf $\mathcal{O}^{\rm an}_X$ of holomorphic functions. For a point $x \in X$, we denote by $\mathcal{O}_{X,x}$ and $\mathcal{O}^{\rm an}_{X,x}$ the corresponding algebraic and analytic local ring at $x$, respectively. By this, $X$ is also regarded as an analytic space.

If $(X,x)$ and $(Y,y)$ are germs of complex algebraic varieties, then they are said to be \emph{analytically isomorphic} if the associated analytic germs $(X^{\rm an},x)$ and $(Y^{\rm an},y)$ are analytically isomorphic.

We shall use the following fundamental characterization: two analytic germs $(X,x)$ and $(Y,y)$ are analytically isomorphic if and only if their completed local rings $\widehat{\mathcal{H}}_{X,x}$ and $\widehat{\mathcal{H}}_{Y,y}$ are isomorphic as $\C$-algebras.

An \emph{algebroid space} is defined by the ideals of the formal power series ring over $\C$. Throughout, we shall denote by $\widehat{(X,x)}$ the formal completion of a complex analytic space germ $(X,x)$, which is an algebroid space.

\begin{thm}[{Hironaka--Rossi, Artin; cf. \cite[Theorem 4.2.3]{Ish2014}}]\label{Thm: Artin's approximation}
Let $(X,x)$ and $(Y,y)$ be germs of analytic spaces. Then the following are equivalent:
\begin{enumerate}[\indent\rm(1)]
\item $(X,x)$ is analytically equivalent to $(Y,y)$;

\item There exists an isomorphism of $\C$-algebras $\widehat{\mathcal{H}}_{X,x}\simeq \widehat{\mathcal{H}}_{Y,y}$;

\item There exists an isomorphism of $\C$-algebras $\mathcal{H}_{X,x}\simeq \mathcal{H}_{Y,y}$.
\end{enumerate}
\end{thm}

\begin{thm}[{Artin's Algebraization Theorem; cf. \cite[Theorem 4.2.4]{Ish2014}}]\label{Thm: Artin's algebraization}
For a germ $(X,x)$ of an analytic space, if $x$ is an isolated singularity, then there exists an algebraic variety $A$ over $\C$ and a point $P \in A$ such that $(X,x)=(A,P)$.
\end{thm}

After Artin's algebraization theorem, as remarked by Ishii, when studying isolated singularities on analytic spaces, one may always assume that they arise from isolated singularities on algebraic varieties. Henceforth, unless otherwise stated, (a germ of) an isolated singularity $(X,x)$ will mean that $X$ is an algebraic variety and $\mathcal{O}_X$ denotes its structure sheaf in the Zariski topology.

\subsection{Generalized cancellation for analytic space germs}\label{subse2.3}
The following two results from \cite{HM1990} are fundamental tools for our subsequent arguments. They provide a unique factorization theory for algebroid spaces and analytic space germs, which will be crucial in identifying and comparing product decompositions.

\begin{defn}
    An algebroid space $V$ is \emph{decomposable} if there are non-trivial algebroid spaces $V_1$ and $V_2$ (i.e., different from the reduced point) such that $V\cong V_1 \times V_2$. And $V$ is called \emph{indecomposable} if it is not decomposable.
\end{defn}

\begin{thm}[{\cite[Theorem 2]{HM1990}}; Unique factorization property of algebroid spaces]\label{Thm: H-M_Unique factorization for algebroid spaces}\mbox{}

For any non-trivial algebroid space $Z$ there exist a unique integer $p$ and non-trivial indecomposable algebroid spaces $Z_1,\ldots,Z_p$, unique up to permutation and isomorphism, such that
\[
Z \cong Z_1 \times \cdots \times Z_p.
\]
\end{thm}

\medskip

The analytic version of this statement remains open in general. However, it does hold if the germ is algebraic (i.e., Nash analytic), i.e., if it can be defined by a power series which in suitable coordinates is algebraic over the ring of polynomials:

\begin{thm}[{cf. \cite[Theorem 3]{HM1990}}; Unique factorization property of algebraic analytic space germs]\label{Thm: H-M_Unique factorization for analytic space germs}\mbox{}

For any non-trivial algebraic analytic space germ $Z$ there exist a unique integer $p$ and non-trivial indecomposable analytic space germs $Z_1,\dots,Z_p$, unique up to permutation and isomorphism, such that
\[
Z \cong Z_1 \times \cdots \times Z_p.
\]
The factors $Z_i$ are algebraic. Moreover, passing to completions, this decomposition coincides with the factorization of $\widehat{Z}$ into indecomposable algebroid spaces.
\end{thm}

\section{\bf Proofs of the Main Results}\label{se3}
\subsection{Exponent rigidity for Brieskorn--Pham varieties}\label{subse3.1}
The following elementary lemma is the key ingredient for the exponent rigidity phenomenon for the Brieskorn--Pham varieties.

\begin{lem}\label{Lem: Key Lemma}
Let $\ba=(a_1,\ldots,a_r)\in\N^r$ and $\bb=(b_1,\ldots,b_s)\in\N^s$ satisfy
\[
1\le a_1\le \cdots\le a_r
\quad\text{and}\quad
1\le b_1\le \cdots\le b_s,
\]
where $1\leq r,s\le n$. Define
\[
R:=\frac{k[X_1,\ldots,X_n]}{(X_1^{a_1},\ldots,X_r^{a_r})}
\qquad\text{and}\qquad
S:=\frac{k[X_1,\ldots,X_n]}{(X_1^{b_1},\ldots,X_s^{b_s})}.
\]
If $R\cong_{k\text{-alg}} S$, then $r=s$ and $a_i=b_i$ for all $1\le i\le r$.
\end{lem}

\begin{proof}
Since $R\cong_{k\text{-alg}} S$, the rings $R$ and $S$ have the same Krull dimensions. Since $\dim(R)=n-r$ and $\dim(S)=n-s$, it follows that $r=s$.

Let $\varphi:R\xrightarrow{\sim}S$ be a $k$-algebra isomorphism. Let $\mathfrak{N}_R$ and $\mathfrak{N}_S$ denote the nilradicals of $R$ and $S$, respectively. Since ring isomorphisms preserve nilpotency, $\varphi(\mathfrak N_R)=\mathfrak N_S$ and hence $\varphi(\mathfrak N_R^m)=\mathfrak N_S^m$
for every $m\ge0$, and therefore $\varphi$ induces isomorphisms $\mathfrak N_R^m/\mathfrak N_R^{m+1} \cong
\mathfrak N_S^m/\mathfrak N_S^{m+1}$. Thus, we obtain
\[
\operatorname{gr}_{\mathfrak{N}_R}(R)
:=
\bigoplus_{m\geq 0}
\frac{\mathfrak{N}_R^m}{\mathfrak{N}_R^{m+1}}
\cong_k
\bigoplus_{m\geq 0}
\frac{\mathfrak{N}_S^m}{\mathfrak{N}_S^{m+1}}
=: \operatorname{gr}_{\mathfrak{N}_S}(S).
\]
as graded $k$-algebras. Therefore, $H_{\operatorname{gr}_{\mathfrak{N}_R}(R)}(t)= H_{\operatorname{gr}_{\mathfrak{N}_S}(S)}(t)$,
where $H_{A}$ denotes the Hilbert series of a $k$-algebra $A$. 

It is well-known that, $$\chi^k(T/I, T/J)(t) :=\sum\limits_{i\geq 0}(-1)^i H_{\Tor^k_i(T/I,T/J)}(t)=\frac{H_{T/I}(t)\cdot H_{T/J}(t)}{H_{k}(t)},$$ as Laurent polynomial in $\Z(\!(t)\!)$, for $T$ being a polynomial ring over $k$ and $I, J \subseteq T$ two ideals of $T$, see \cite{AB1993}. Since all $k$-vector spaces are flat $k$-modules, $\Tor^k_i(T/I,T/J)=0$ for $i>0$. Thus, the above formula yields that 
\begin{equation}\label{eq:mult_of_Hilbert_series}
    H_{T/I\otimes_k T/J}(t)=H_{T/I}(t)\cdot H_{T/J}(t),
\end{equation} 
as $H_{k}(t)=1 \in \Z(\!(t)\!)$.

Observe that $\mathfrak{N}_R=(\bar{X}_1,\ldots,\bar{X}_r)$, where $\bar{X}_i$ denotes the image of $X_i$ in $R$. Therefore the quotient $\mathfrak N_R^m/\mathfrak N_R^{m+1}$ naturally gets an $R/\mathfrak{N}_R$-module structure, and $R/\mathfrak{N}_R \cong k[X_{r+1},\ldots,X_n]$. Moreover, it is evident that $\mathfrak N_R^m/\mathfrak N_R^{m+1}$ is a free $k[X_{r+1},\ldots,X_n]$-module with basis consisting of the images of the monomials
\[
\{\bar{X}_1^{i_1}\cdots \bar{X}_r^{i_r}
\mid
i_1+\cdots+i_r=m,\;
0\le i_j<a_j
\}.
\]
Consequently,
\[
\operatorname{gr}_{\mathfrak N_R}(R)
\cong
k[X_{r+1},\ldots,X_n]
\otimes_k
\frac{k[U_1,\ldots,U_r]}
{(U_1^{a_1},\ldots,U_r^{a_r})},
\]
where each $U_i$ is homogeneous of degree $1$. Similarly,
\[
\operatorname{gr}_{\mathfrak{N}_S}(S)
\cong
k[X_{r+1},\ldots,X_n]
\otimes_k
\frac{k[V_1,\ldots,V_r]}
{(V_1^{b_1},\ldots,V_r^{b_r})},
\]
where each $V_i$ is homogeneous of degree $1$.

Since Hilbert series are multiplicative under tensor products of graded $k$-algebras, we obtain
\[
H_{\operatorname{gr}_{\mathfrak{N}_R}(R)}(t)
=
\prod_{i=1}^{r}
(1+t+\cdots+t^{a_i-1})
\cdot
\left(\sum_{\ell\geq 0} t^\ell\right)^{n-r},
\]
and
\[
H_{\operatorname{gr}_{\mathfrak{N}_S}(S)}(t)
=
\prod_{i=1}^{r}
(1+t+\cdots+t^{b_i-1})
\cdot
\left(\sum_{\ell\geq 0} t^\ell\right)^{n-r}.
\]
Since these two Hilbert series are equal, it follows that
\[
\prod_{i=1}^{r}
(1+t+\cdots+t^{a_i-1})
=
\prod_{i=1}^{r}
(1+t+\cdots+t^{b_i-1}),
\]
and hence
\begin{equation}\label{eq:key}
\prod_{i=1}^{r}(1-t^{a_i})
=
\prod_{i=1}^{r}(1-t^{b_i}).
\end{equation}

Let $p$ and $q$ be the largest indices such that $a_1=\cdots=a_p$ and $b_1=\cdots=b_q$, respectively, for some $1\leq p,\, q\leq r$. Expanding both sides of \eqref{eq:key}, we obtain
\begin{equation}\label{eq:key1}
1-pt^{a_1}+P(t)=1-qt^{b_1}+Q(t),
\end{equation}
where \(P(t),Q(t)\in\Z[t]\) satisfy $\operatorname{ord}_t(P)>a_1$ and $\operatorname{ord}_t(Q)>b_1$, i.e., $P(t)$ and $Q(t)$ contain only terms of degree strictly greater than $a_1$ and $b_1$, respectively. Comparing the lowest nonconstant terms shows that $a_1=b_1$ and $p=q$. Therefore $a_i=b_i$ for $1\leq i\leq p$. Cancelling the common factors \((1-t^{a_1})^p=(1-t^{b_1})^p\) from \eqref{eq:key}, we obtain
\begin{equation}\label{eq:key2}
\prod_{i=p+1}^{r}(1-t^{a_i})
=
\prod_{i=p+1}^{r}(1-t^{b_i}).
\end{equation}
Repeating the same argument inductively yields $a_i=b_i$ for $1\leq i\leq r$. Hence $r=s$ and $a_i=b_i$ for all $1\leq i\leq r$. This completes the proof.
\end{proof}

\begin{rem}\label{Rem 1}
    The same conclusion in Lemma \ref{Lem: Key Lemma} holds if we consider $R$ and $S$ over $k=\C$, as the same quotients of either the convergent power series or formal power series, instead of polynomial algebras.
\end{rem}

We are now in a position to establish exponent rigidity for Brieskorn--Pham varieties.

\begin{thm}\label{Thm: Exponent rigidity of B-P variety}
    Let $V_\ba=\Spec B_\ba$ and $V_\bb=\Spec B_\bb$ be Brieskorn--Pham varieties over a field $k$ of characteristic $0$ corresponding to exponent tuples $\ba,\bb\in(\N_{\ge2})^n$ for $n\geq 2$. Then $V_\ba\cong_k V_\bb$ if and only if $\ba\sim \bb$, i.e., $\ba=\bb$ up to permutation of entries.
\end{thm}

\begin{proof}
The implication ``$\ba\sim\bb \Rightarrow V_\ba\cong_k V_\bb$'' is obvious by permuting the coordinates. We prove the converse. Assume that $V_\ba\cong_k V_\bb$. Let $\bar{k}$ be an algebraic closure of $k$. Base-changing the given isomorphism to $\bar{k}$ yields $V_{\ba,\bar{k}}\cong_{\bar{k}} V_{\bb,\bar{k}}$, where $$V_{\ba,\bar{k}}:=V_\ba\times_{\Spec k} \Spec \bar{k}=\Spec B_{\ba,\bar{k}}= \Spec(B_\ba\otimes_k \bar{k})=\Spec\!\left(
\frac{\bar{k}[X_1,\ldots,X_n]}
{(X_1^{a_1}+\cdots+X_n^{a_n})}
\right),$$ and the same for $V_{\bb,\bar{k}}$. Let $f_\ba:=X_1^{a_1}+\cdots+X_n^{a_n}$, and $f_\bb:=X_1^{b_1}+\cdots+X_n^{b_n}$. Since the scheme-theoretic singular locus is preserved under isomorphisms, the above $\bar{k}$-isomorphism induces an isomorphism $\Sing(V_{\ba,\bar{k}})\cong_{\bar{k}} \Sing(V_{\bb,\bar{k}})$ as $\bar{k}$-schemes.

Since $V_{\ba, \bar{k}}$ is a hypersurface defined over $\bar{k}$, the Jacobian criterion yields, 
\[
\Sing(V_{\ba,\bar{k}})
=
\Spec\!\left(
\frac{\bar{k}[X_1,\ldots,X_n]}
{(f_\ba,\partial_1f_\ba,\ldots,\partial_nf_\ba)}
\right),
\]
where $\partial_i\equiv\frac{\partial}{\partial X_i}$ for all $1 \le i \le n$. Since $\partial_i f_\ba=a_iX_i^{a_i-1}$ for every $1\le i\le n$
and $\operatorname{char}(k)=\operatorname{char}(\bar{k})=0$, we have
\[
f_\ba
=
\sum_{i=1}^n
{a_i}^{-1}
X_i\frac{\partial f_\ba}{\partial X_i}
\in
(\partial_1f_\ba,\ldots,\partial_nf_\ba).
\]
Hence $(f_\ba,\partial_1f_\ba,\ldots,\partial_nf_\ba)
=
(X_1^{a_1-1},\ldots,X_n^{a_n-1}) \subseteq \bar{k}[X_1,\ldots,X_n]$, and therefore
\[
\Sing(V_{\ba,\bar{k}})
=
\Spec\!\left(
\frac{\bar{k}[X_1,\ldots,X_n]}
{(X_1^{a_1-1},\ldots,X_n^{a_n-1})}
\right) \quad \text{and} \quad \Sing(V_{\bb,\bar{k}})
=
\Spec\!\left(
\frac{\bar{k}[X_1,\ldots,X_n]}
{(X_1^{b_1-1},\ldots,X_n^{b_n-1})}
\right).
\]
Since $\Sing(V_{\ba,\bar{k}})\cong_{\bar{k}}\Sing(V_{\bb,\bar{k}})$ as affine $\bar{k}$-schemes, it follows that
\[
\Gamma(\Sing(V_{\ba,\bar{k}}),\mathcal O_{\Sing(V_{\ba,\bar{k}})})
\cong_{\bar{k}}
\Gamma(\Sing(V_{\bb,\bar{k}}),\mathcal O_{\Sing(V_{\bb,\bar{k}})}).
\]
That is,
\[
\frac{\bar{k}[X_1,\ldots,X_n]}
{(X_1^{a_1-1},\ldots,X_n^{a_n-1})}
\cong_{\bar{k}\text{-alg}}
\frac{\bar{k}[X_1,\ldots,X_n]}
{(X_1^{b_1-1},\ldots,X_n^{b_n-1})}.
\]
Therefore, Lemma~\ref{Lem: Key Lemma} yields
\[
(a_1-1,\ldots,a_n-1)\sim(b_1-1,\ldots,b_n-1),
\]
and hence $\ba\sim\bb$. This completes the proof.
\end{proof}

As an application of Theorem~\ref{Thm: Exponent rigidity of B-P variety}, we obtain the following analytic rigidity result over $\C$. Note that this may also be deduced from the main result of Yoshinaga--Suzuki \cite{YS1978}.

\begin{cor}\label{Cor: Exponent rigidity of B-P variety in analytic setting over complex numbers}
Let $V_\ba$ and $V_\bb$ be two Brieskorn--Pham varieties over $\C$ corresponding to two tuples $\ba,\bb\in(\N_{\geq 2})^n$ for $n\geq 2$. Then
\[
\widehat{V}_\ba \cong_{\rm bihol} \widehat{V}_\bb
\iff
V_\ba^{\rm an} \cong_{\rm bihol} V_\bb^{\rm an}
\iff
\ba\sim\bb.
\]
\end{cor}

\begin{proof}
The first equivalence follows from Theorem~\ref{Thm: Artin's approximation}. Assume that $V_\ba^{\rm an} \cong_{\rm bihol} V_\bb^{\rm an}$. Then the corresponding analytic local algebras $B_\ba^{\rm an}$ and $B_\bb^{\rm an}$ are isomorphic. Retracing the arguments in the proofs of Lemma~\ref{Lem: Key Lemma} (see Remark \ref{Rem 1}) and Theorem~\ref{Thm: Exponent rigidity of B-P variety}, we obtain $\ba\sim\bb$. The converse implication is immediate.
\end{proof}

\subsection{Generalized Zariski cancellation for good $\C^*$-varieties with a unique singularity}\label{subse3.2}

We begin with the following elementary observation, which will be used repeatedly in the proof of the main theorem of this subsection.

\begin{lem}\label{Lem: Isolated implies indecomposable}
Let $X$ be a connected algebroid space of positive dimension with an isolated singularity. Then $X$ is indecomposable.
\end{lem}

\begin{proof}
Suppose that $X \cong Y \times Z$ for algebroid spaces $Y$ and $Z$. Since $\dim(X)>0$, at least one of the factors has positive dimension. Without loss of generality, assume that $\dim(Y)>0$. Clearly
\[
\Sing(X)
=
(\Sing(Y)\times Z)
\cup
(Y\times \Sing(Z)),
\]
and hence
\[
\dim(\Sing(X))
=
\max\{
\dim(\Sing(Y))+\dim(Z),
\,
\dim(Y)+\dim(\Sing(Z))
\}.
\]
Since $X$ has an isolated singularity, we have $\dim(\Sing(X))=0$. As $\dim(Y)>0$, the equality above forces
\[
\dim(Z)=0
\qquad\text{and}\qquad
\Sing(Z)=\varnothing.
\]
Thus $Z$ is a smooth zero-dimensional connected algebroid space, hence a reduced point. Therefore $X \cong Y$, showing that $X$ is indecomposable.
\end{proof}

\begin{rem}
In the above proof, we adopt the convention that $\dim(\varnothing)=-\infty$.
\end{rem}

We now prove the key technical result from which the main theorem of this subsection follows.

\begin{thm}\label{Thm: Our main cancellation theorem}
Let $X$ and $Y$ be connected (not necessarily irreducible) algebraic varieties over $\C$, each having a unique singular point, say $P\in X$ and $Q\in Y$. Let $Z$ be a separated scheme over $\C$ (not necessarily connected) having a smooth point such that
\[
X\times Z \cong_{\C} Y\times Z.
\]
Then the analytic germs $(X,P)$ and $(Y,Q)$ are analytically isomorphic.
\end{thm}

\begin{proof}
Let $\varphi:X\times Z \xrightarrow{\sim} Y\times Z$ be an isomorphism. Put $n:=\dim(X)=\dim(Y)$.

Choose a smooth point $z_0\in Z$. Since $z_0$ is smooth, it belongs to a unique irreducible component of $Z$. Let $m$ denote the dimension of that irreducible component. Thus, $(Z,z_0)\cong (\C^m,\oo)$.

\medskip
\textbf{Claim.} {\it There exist a non-negative integer $\ell \leq \lfloor \frac{m}{n}\rfloor$ and a sequence of points ${z_0,\ldots,z_\ell}\subseteq Z$ satisfying the following two properties:
\begin{enumerate}[\indent\rm(1)]
\item for all $0\leq i\leq \ell-1$, $\varphi(P,z_i)=(y_{i+1},z_{i+1})$, where $y_{i+1}\in Y^{\rm sm}$, and $\varphi(P,z_\ell)=(Q,w_\ell)$ for some $w_\ell\in Z$, i.e., $\ell$ is the least non-negative integer such that the projection of $\varphi(P,z_\ell)$ onto $Y$ is $Q$;

\item for all $0\leq i\leq \ell$,
\begin{equation}\label{eqn1}
\widehat{(Z,z_i)}
\cong
\underbrace{\widehat{(X,P)}\times \cdots \times \widehat{(X,P)}}_{i\text{ copies}}
\times
\widehat{(\C^{m-in},\oo)}.
\end{equation}
\end{enumerate}
}

\smallskip

\textit{Proof of the claim.}
Clearly $z_0$ satisfies \eqref{eqn1}.

If $\varphi(P,z_0)=(Q,w)$ for some $w\in Z$, then we choose $\ell=0$ and the claim follows. Thus, assume that the projection of $\varphi(P,z_0)$ onto $Y$ is not $Q$.

Let $i>0$ and suppose that ${z_0,\ldots,z_{i-1}}$ has been chosen satisfying:

\begin{enumerate}[\indent\rm(1)]
\item for all $0\leq j\leq i-2$, $\varphi(P,z_j)=(y_{j+1},z_{j+1})$, where $y_{j+1}\in Y^{\rm sm}$;

\item for all $0\leq j\leq i-1$,
\[
\widehat{(Z,z_j)}
\cong
\underbrace{\widehat{(X,P)}\times \cdots \times \widehat{(X,P)}}_{j\text{ copies}}
\times
\widehat{(\C^{m-jn},\oo)}.
\]
\end{enumerate}

Exactly one of the following two cases can occur.

\medskip

\textbf{Case 1.} \emph{$\varphi(P,z_{i-1})=(Q,w_{i-1})$ for some $w_{i-1}\in Z$.}

\smallskip

In this case, we set $\ell=i-1$, and the required sequence has been obtained.

\medskip

\textbf{Case 2.} \emph{The projection of $\varphi(P,z_{i-1})$ onto $Y$ is a smooth point, say $y_i\in Y^{\rm sm}$.}

\smallskip

Define $z_i$ to be the projection of $\varphi(P,z_{i-1})$ onto $Z$. Since $\varphi$ is an isomorphism, it induces an isomorphism of algebroid spaces
\[
\widehat{(X,P)}
\times
\widehat{(Z,z_{i-1})}
\cong
\widehat{(Y,y_i)}
\times
\widehat{(Z,z_i)}.
\]

Using the induction hypothesis, we obtain
\begin{equation}\label{eqn2}
\underbrace{\widehat{(X,P)}\times \cdots \times \widehat{(X,P)}}_{i\text{ copies}}
\times
\widehat{(\C^{m-n(i-1)},\oo)}
\cong
\widehat{(\C^n,\oo)}
\times
\widehat{(Z,z_i)}.
\end{equation}

We claim that $m\geq in$. Suppose otherwise. Then $m<in$, and therefore $m-n(i-1)<n$. Applying Hauser--M\"uller's cancellation theorem \cite[Theorem 1]{HM1990} to \eqref{eqn2}, we obtain
\[
\underbrace{\widehat{(X,P)}\times \cdots \times \widehat{(X,P)}}_{i\text{ copies}}
\cong
\widehat{(\C^{in-m},\oo)}
\times
\widehat{(Z,z_i)}.
\]
Since $in-m>0$, the right-hand side contains $\widehat{(\C,0)}$ as an indecomposable factor. On the other hand, the left-hand side does not contain $\widehat{(\C,0)}$ as an indecomposable factor because $\widehat{(X,P)}$ is indecomposable by Lemma \ref{Lem: Isolated implies indecomposable}, and $\widehat{(X,P)}\not\cong \widehat{(\C,0)}$ since $P$ is a singular point of $X$. This contradicts the uniqueness of factorization theorem for algebroid spaces due to Hauser--M\"uller (Theorem \ref{Thm: H-M_Unique factorization for algebroid spaces}). Hence $m\geq in$, as claimed. Applying Hauser--M\"uller's cancellation theorem once again to \eqref{eqn2}, we obtain
\[
\underbrace{\widehat{(X,P)}\times \cdots \times \widehat{(X,P)}}_{i\text{ copies}}
\times
\widehat{(\C^{m-in},\oo)}
\cong
\widehat{(Z,z_i)},
\]
whence $z_i$ satisfies \eqref{eqn1}.

If Case $1$ never occurs, then Case $2$ occurs for every positive integer $i$. By the above argument, this would imply that $m\geq in$ for every positive integer $i$, which is impossible for $i>\frac{m}{n}$. Therefore Case $1$ must occur for some non-negative integer $\ell\leq \lfloor \frac{m}{n}\rfloor$. This proves the claim.
\qed

\medskip

\textit{Proof of Theorem \ref{Thm: Our main cancellation theorem} continued.}
Since $\varphi(P,z_\ell)=(Q,w_\ell)$, the isomorphism $\varphi$ induces an isomorphism of algebroid spaces
\[
\widehat{(X,P)}
\times
\widehat{(Z,z_\ell)}
\cong
\widehat{(Y,Q)}
\times
\widehat{(Z,w_\ell)}.
\]

Using \eqref{eqn1}, we obtain
\[
\underbrace{\widehat{(X,P)}\times \cdots \times \widehat{(X,P)}}_{\ell+1\text{ copies}}
\times
\widehat{(\C^{m-n\ell},\oo)}
\cong
\widehat{(Y,Q)}
\times
\widehat{(Z,w_\ell)}.
\]

By Lemma \ref{Lem: Isolated implies indecomposable}, both $\widehat{(X,P)}$ and $\widehat{(Y,Q)}$ are indecomposable algebroid spaces. By the uniqueness of factorization theorem for algebroid spaces due to Hauser--M\"uller (Theorem \ref{Thm: H-M_Unique factorization for algebroid spaces}), the indecomposable factor $\widehat{(Y,Q)}$ must occur among the indecomposable factors on the left-hand side. Now, the indecomposable factors occurring on the left-hand side are the $\ell+1$ copies of $\widehat{(X,P)}$ and, if $m-n\ell>0$, the $m-n\ell$ copies of $\widehat{(\C,0)}$. Since $Q$ is a singular point of $Y$, we have
$\widehat{(Y,Q)}\not\cong\widehat{(\C,\oo)}$. Therefore, $\widehat{(Y,Q)}\cong \widehat{(X,P)}$.

Finally, by Artin's approximation theorem (Theorem \ref{Thm: Artin's approximation}), this formal isomorphism algebraizes to an analytic isomorphism of complex analytic germs. Hence, $(X,P)\cong (Y,Q)$. This completes the proof.
\end{proof}

As an immediate consequence of Theorem \ref{Thm: Our main cancellation theorem}, we obtain the following result for varieties with isolated singularities.

\begin{cor}\label{Cor: Cancellation theorem for isolated singularities}
Let $X$ and $Y$ be connected (not necessarily irreducible) algebraic varieties over $\C$ having only isolated singularities, and let $Z$ be a separated $\C$-scheme (not necessarily connected) having a smooth point. Suppose that
\[
X\times Z \cong_{\C} Y\times Z.
\]
Then, for every singular point $P\in \Sing(X)$, there exists a singular point $Q\in \Sing(Y)$ such that the analytic germs $(X,P)$ and $(Y,Q)$ are analytically isomorphic.
\end{cor}

\begin{proof}
Fix a singular point $P\in \Sing(X)$. Repeating verbatim the proof of Theorem \ref{Thm: Our main cancellation theorem}, replacing the unique singular point $Q$ of $Y$ by a singular point of $Y$, we obtain a singular point $Q\in \Sing(Y)$ such that $(X,P)\cong (Y,Q)$ as complex analytic germs.
\end{proof}

\begin{rem}
Applying Corollary \ref{Cor: Cancellation theorem for isolated singularities} to the inverse isomorphism $\varphi^{-1}: Y\times Z \xrightarrow{\sim} X\times Z$, it follows that the sets of analytic isomorphism classes of isolated singularity germs occurring on $X$ and $Y$ coincide. More precisely, for any analytic singularity type $\tau$, let
\[
\Sigma_X(\tau):=\{P\in \Sing(X)\mid (X,P)\text{ is of type }\tau\}
\quad \text{and} \quad
\Sigma_Y(\tau):=\{Q\in \Sing(Y)\mid (Y,Q)\text{ is of type }\tau\}.
\]
Then
\[
\Sigma_X(\tau)\neq \varnothing
\quad\iff\quad
\Sigma_Y(\tau)\neq \varnothing.
\]
In other words, $X$ and $Y$ have exactly the same strata of analytic singularity types.
\end{rem}

The following theorem of R. V. Gurjar provides an algebraization result for complex affine varieties admitting good $\C^*$-actions. Roughly speaking, it shows that analytic equivalence of the corresponding germs at the vertices already implies algebraic equivalence of the underlying varieties.

\begin{lem}[{cf. \cite[Theorem 2]{Gur2020}}]\label{Lem: Gurjar's result}
Let $R$ and $S$ be positively graded affine domains over $\C$. Let $V:=\Spec R$ and $W:=\Spec S$ be the corresponding complex affine varieties admitting good $\C^*$-actions with vertices $p$ and $q$, respectively (i.e., the closed points corresponding to the irrelevant maximal ideals of $R$ and $S$, respectively). If the complex analytic germs of $V$ and $W$ at $p$ and $q$, respectively, are isomorphic, then $V$ and $W$ are isomorphic as affine varieties (not necessarily as $\C^*$-varieties).
\end{lem}

We now combine Theorem~\ref{Thm: Our main cancellation theorem} with Gurjar's algebraization theorem (Lemma~\ref{Lem: Gurjar's result}) to deduce a generalized Zariski cancellation theorem for complex affine varieties admitting good $\C^*$-actions and having a unique singularity. This yields one of the main results of the article.

\begin{thm}\label{Thm: Our main cancellation theorem for C*-varieties}
Let $R$ and $S$ be positively graded affine domains over $\C$. Let $V:=\Spec R$ and $W:=\Spec S$ be the corresponding complex affine irreducible varieties admitting good $\C^*$-actions with vertices $p$ and $q$, respectively. Assume that $p$ and $q$ are the unique singularities of $V$ and $W$, respectively. Let $Z$ be a separated scheme over $\C$ (not necessarily connected) having a smooth point such that
\[
V\times Z \cong_{\C} W\times Z.
\]
Then $V$ and $W$ are isomorphic as complex affine varieties (not necessarily as $\C^*$-varieties).
\end{thm}

\begin{proof}
By Theorem \ref{Thm: Our main cancellation theorem}, the isomorphism
\[
V\times Z \xrightarrow{\sim} W\times Z
\]
induces an analytic isomorphism of germs $(V,p)\cong (W,q)$. Since $V$ and $W$ admit good $\C^*$-actions with vertices $p$ and $q$, respectively, Gurjar's result (Lemma \ref{Lem: Gurjar's result}) implies that this analytic isomorphism extends to an algebraic isomorphism $V\cong_{\C} W$. Hence $V$ and $W$ are isomorphic as complex affine varieties.
\end{proof}

\begin{rem}
Observe that if $Z$ is further assumed to be smooth in Theorem~\ref{Thm: Our main cancellation theorem for C*-varieties}, then it suffices to assume that $p$ is the unique singular point of $V$. Indeed, since $Z$ is smooth,
\[
\Sing(V\times Z)=\Sing(V)\times Z=\{p\}\times Z
\quad
\text{and}\quad 
\Sing(W\times Z)=\Sing(W)\times Z.
\]
Hence the isomorphism $V\times Z\xrightarrow{\sim} W\times Z$ induces an isomorphism $Z\xrightarrow{\sim} \Sing(W)\times Z$.
Comparing dimensions, we obtain $\dim(\Sing(W))=0$. Since $\Sing(W)$ is a closed subvariety of the affine variety $W$, it follows that $\Sing(W)$ is a finite set of points. If $\Sing(W)=\{s_1,\ldots,s_r\}$ consisting of $r$ points ($r>0$), then
\[
\Sing(W)\times Z=\bigsqcup_{i=1}^{r} \left(\{s_i\}\times Z\right).
\]
Since $\Sing(W)\times Z\cong Z$, comparison of the numbers of connected components yields $r=1$. Thus $W$ has a unique singular point. Since the vertex $q$ is fixed by the good $\C^*$-action, this unique singular point must be $q$. Therefore the conclusion of Theorem~\ref{Thm: Our main cancellation theorem for C*-varieties} remains valid under the weaker assumption that only $p$ is the unique singular point of $V$.
\end{rem}

\subsection{Generalized Zariski cancellation for Brieskorn--Pham varieties}\label{subse3.3}

We now conclude this article by proving the generalized Zariski cancellation theorem for Brieskorn--Pham varieties.

\begin{thm}\label{Thm: Our main cancellation theorem for B-P varieties}
    Let $n\geq 2$, and let $V_\ba$ and $V_\bb$ be two Brieskorn--Pham varieties defined over $\C$ corresponding to two tuples $\ba,\bb\in(\N_{\geq 2})^n$. Then, for any separated $\C$-scheme $Z$ (not necessarily connected) having a smooth point, the following are equivalent:
    \begin{enumerate}[\indent\rm(1)]
        \item $V_\ba\times Z\cong_{\C} V_\bb\times Z$;
        \item $V_\ba\cong_{\C} V_\bb$;
        \item $V_\ba$ and $V_\bb$ are isomorphic as $\C^*$-varieties;
        \item $\ba \sim \bb$;
    \end{enumerate}
\end{thm}

\begin{proof}
With the earlier notations, let $V_\ba=\Spec B_\ba$ and $V_\bb=\Spec B_\bb$, where $B_\ba$ and $B_\bb$ are the corresponding Brieskorn--Pham hypersurface rings, and let $\oo_\ba$ and $\oo_\bb$ denote the unique singular points of $V_\ba$ and $V_\bb$, respectively. Recall from Section~\ref{subse2.1} that $B_\ba$ and $B_\bb$ are positively graded affine $\C$-domains. Consequently, $V_\ba$ and $V_\bb$ admit good $\C^*$-actions with vertices $\oo_\ba$ and $\oo_\bb$, respectively.

We first establish $\textup{(4)}\implies\textup{(3)}$. Assume that $\ba\sim\bb$. Then there exists a permutation $\sigma\in S_n$ such that $b_i=a_{\sigma(i)}$ for all $i=1,\ldots,n$. Let
\[
B_\ba=
\frac{\C[X_1,\ldots,X_n]}
{(X_1^{a_1}+\cdots+X_n^{a_n})}
\quad\text{and}\quad
B_\bb=
\frac{\C[Y_1,\ldots,Y_n]}
{(Y_1^{b_1}+\cdots+Y_n^{b_n})},
\]
and let $x_i$ and $y_i$ denote the images of $X_i$ and $Y_i$ in $B_\ba$ and $B_\bb$, respectively, for $i=1,\ldots,n$. Consider the $\C$-algebra homomorphism $\varphi:B_\bb\longrightarrow B_\ba$
defined by $\varphi(y_i)=x_{\sigma(i)}$ for all $i=1,\ldots,n$. Since $b_i=a_{\sigma(i)}$, we have
\[
\varphi\!\left(\sum_{i=1}^{n}y_i^{b_i}\right)
=
\sum_{i=1}^{n}x_{\sigma(i)}^{a_{\sigma(i)}}
=
\sum_{j=1}^{n}x_j^{a_j},
\]
and hence $\varphi$ is well-defined. Its inverse is induced by the inverse permutation $\sigma^{-1}$, so $\varphi$ is a $\C$-algebra isomorphism.

Let $N=\operatorname{lcm}(a_1,\ldots,a_n) =
\operatorname{lcm}(b_1,\ldots,b_n)$. For each $i=1,\ldots,n$, put $w_i=\frac{N}{a_i}$ and $w_i'=\frac{N}{b_i}$. By Section~\ref{subse2.1}, the good $\C^*$-actions on $V_\ba$ and $V_\bb$ are induced by the positive gradings on $B_\ba$ and $B_\bb$ determined by $\deg(x_i)=w_i$ and $\deg(y_i)=w_i'$, respectively. Since $b_i=a_{\sigma(i)}$, we obtain 
\[
\deg(y_i)=w_i'=\frac{N}{b_i}=\frac{N}{a_{\sigma(i)}}=w_{\sigma(i)}=\deg(x_{\sigma(i)})
\]
for every $i=1,\ldots,n$. Therefore, $\varphi$ preserves the degrees, and thus it is a graded isomorphism of graded $\C$-algebras. Moreover, for every $t\in\C^*$ and every generator $y_i$ of $B_\bb$, we have
\[
\varphi(t\cdot y_i)=\varphi(t^{w_i'}y_i)=t^{w_i'}x_{\sigma(i)}=t^{w_{\sigma(i)}}x_{\sigma(i)}=t\cdot x_{\sigma(i)}=t\cdot\varphi(y_i).
\]
Since $y_1,\ldots,y_n$ generate $B_\bb$ as a $\C$-algebra, it follows that $\varphi$ intertwines the corresponding $\C^*$-actions on $B_\bb$ and $B_\ba$. Consequently, the induced isomorphism
\[
\varphi^*:V_\ba\longrightarrow V_\bb
\]
is $\C^*$-equivariant. Thus $V_\ba$ and $V_\bb$ are isomorphic as $\C^*$-varieties.

The implications $\rm(3)\implies\rm(2)\implies\rm(1)$ are straightforward. Finally, the implication $\textup{(1)}\implies\textup{(4)}$ follows from Theorem~\ref{Thm: Our main cancellation theorem} together with Corollary~\ref{Cor: Exponent rigidity of B-P variety in analytic setting over complex numbers}. This completes the proof.
\end{proof}

\section*{\bf Acknowledgements}
Both authors would like to thank Prof. R. V. Gurjar and Prof. A. J. Parameswaran for several helpful discussions. This collaboration began while both authors were postdoctoral fellows at the Chennai Mathematical Institute (CMI), India.

The first-named author acknowledges financial support from the National Board for Higher Mathematics (NBHM), Department of Atomic Energy, Government of India, through a Postdoctoral Fellowship during his tenure at CMI. He also acknowledges current financial support from the Department of Science and Technology, Government of India, through the INSPIRE Faculty Fellowship (Reference No. DST/INSPIRE/04/2024/003379, Faculty Registration No. IFA24--MA 213).

The second-named author acknowledges 
the support of the project 
``Singularities and Applications'' 
- CF 132/31.07.2023 funded by the European Union - NextGenerationEU - through Romania's National Recovery and Resilience Plan, and the support of the grant CNRS-INSMI-IEA-329.

\section*{\bf Data Availability Statement}

Data sharing does not apply to this article, as no datasets were generated or analysed during the current study.

\section*{\bf Declarations}

The authors declare that they have no competing interests. No additional funding was received for this work other than that acknowledged above.

\bibliographystyle{alpha}
\bibliography{ref}
\end{document}